\documentclass[12pt]{amsart} %
\usepackage{eurosym, color, scalerel}

\newcommand{\rd}{\mathrm{d}}

\newcommand{\abs}[1]{\left|#1\right|}

\usepackage{DKstyle}

\newcommand{\vertiii}[1]{{\left\vert\kern-0.25ex\left\vert\kern-0.25ex\left\vert #1
    \right\vert\kern-0.25ex\right\vert\kern-0.25ex\right\vert}}
\theoremstyle{plain}
\makeatletter
\newcommand*{\rom}[1]{\expandafter\@slowromancap\romannumeral #1@}
\makeatother
\newtheorem*{thm*}{Theorem}

\usepackage{enumitem}

\usepackage{caption}
\usepackage[labelfont=rm]{subcaption}
\usepackage{bm}
\usepackage[pagebackref=true, colorlinks]{hyperref}

\hypersetup{pdffitwindow=true,linkcolor=blue,citecolor=blue,urlcolor=blue,menucolor=blue}

\usepackage{comment}


\date{\today}%
\dedicatory{}%
\commby{}%

\title{Sign Changes Along Geodesics of Modular Forms}
\author{Dubi Kelmer}
\thanks{Kelmer is partially supported by NSF CAREER grant DMS-1651563.}
\email{kelmer@bc.edu}
\address{Department of Mathematics, Boston College, Boston, Massachusetts, United States}
\author{Alex Kontorovich}
\thanks{Kontorovich is partially supported by NSF grant DMS-2302641, BSF grant 2020119, and a Simons Fellowship.}
\email{alex.kontorovich@rutgers.edu}
\address{Department of Mathematics, Rutgers University, New Brunswick, New Jersey, United States}
\author{Christopher Lutsko}
\email{clutsko@uh.edu}
\address{Department of Mathematics, University of Houston, Houston, Texas, United States}
\date{\today}

\begin{document}

\maketitle
\begin{abstract}
\noindent    
Given a compact segment, $\beta$, of a cuspidal geodesic on the modular surface, we study the number of sign changes of cusp forms and Eisenstein series along $\beta$.
Recent work of Ki \cite{Ki2023} a sharp lower bound the number of sign changes was given assuming sup norm bounds for these forms as well as the  Lindel\"of hypothesis for the corresponding $L$-funcitons. The main result of this work is to obtain the same result for Eisenstein series without the Lindel\"of hypothesis.

\end{abstract}

\section{Introduction}

Let $\Gamma < \SL_2(\R)$ 
be a discrete, cofinite group
acting on the upper half-plane $\bH$ by fractional linear transformations. 
Given a real-valued automorphic function $f : \Gamma \bk \bH \to \R$, we denote by $Z_f$ its zero set, which separates the space into connected nodal domains. A key question in the analysis of $f$ is to consider the number of nodal domains.  Of particular interest, with applications to quantum chaos, is to study the number of nodal domains of 
eigenfunctions of the hyperbolic Laplace-Beltrami operator, 
as the eigenvalue goes to infinity.  Henceforth we work specifically with the modular group $\Gamma : = \SL_2(\Z)$; the proofs below can be generalized to congruence subgroups, as long as they include reflection symmetries. 

Recall the spectral decomposition of $L^2(\Gamma \bk \bH)$ into cusp forms and Eisenstein series. 
The Eisenstein series for the modular group $\Gamma
$ is given by
\begin{align*}
    E(z,s) : = \sum_{\gamma \in \Gamma_\infty \bk \Gamma} \Im(\gamma z)^s,
\end{align*}
where $\G_\infty$ is the stabilizer of $\infty$ in $\G$.
This series converges absolutely for $\Re(s)>1$, and has meromorphic continuation for all $s\in \C$. For $s=\frac12+it$, the function 
\begin{align*}
    E_t(z):=E(z,\frac12+it)
\end{align*} 
is an eigenfunction of the Laplacian (as well as all Hecke operators), and has Laplace eigenvalue $\gl=\frac14+t^2$. We include an image of the Eisenstein series in Figure \ref{fig:Eis}.

Moreover, a \emph{Maass cusp form} is a function $\phi: \bH \to \R$ satisfying
\begin{enumerate}[label = (\roman*)]
    \item  $\Delta \phi +\lambda \phi=0, \qquad \lambda = \lambda_\phi >0$,
    \item $\phi(\gamma z)=\phi(z), \qquad \gamma \in \Gamma$,
    \item and $\phi\in L^2 (\Gamma \bk \bH)$ with $L^2$ norm $1$. 
\end{enumerate}
Given such a cusp form $\phi$, we write its eigenvalue as $\lambda_\phi =\frac{1}{4} +  t_\phi^2$.

  \begin{figure}[ht!]
  \begin{center}    
    \vstretch{0.7}{\includegraphics[width=0.5\textwidth]{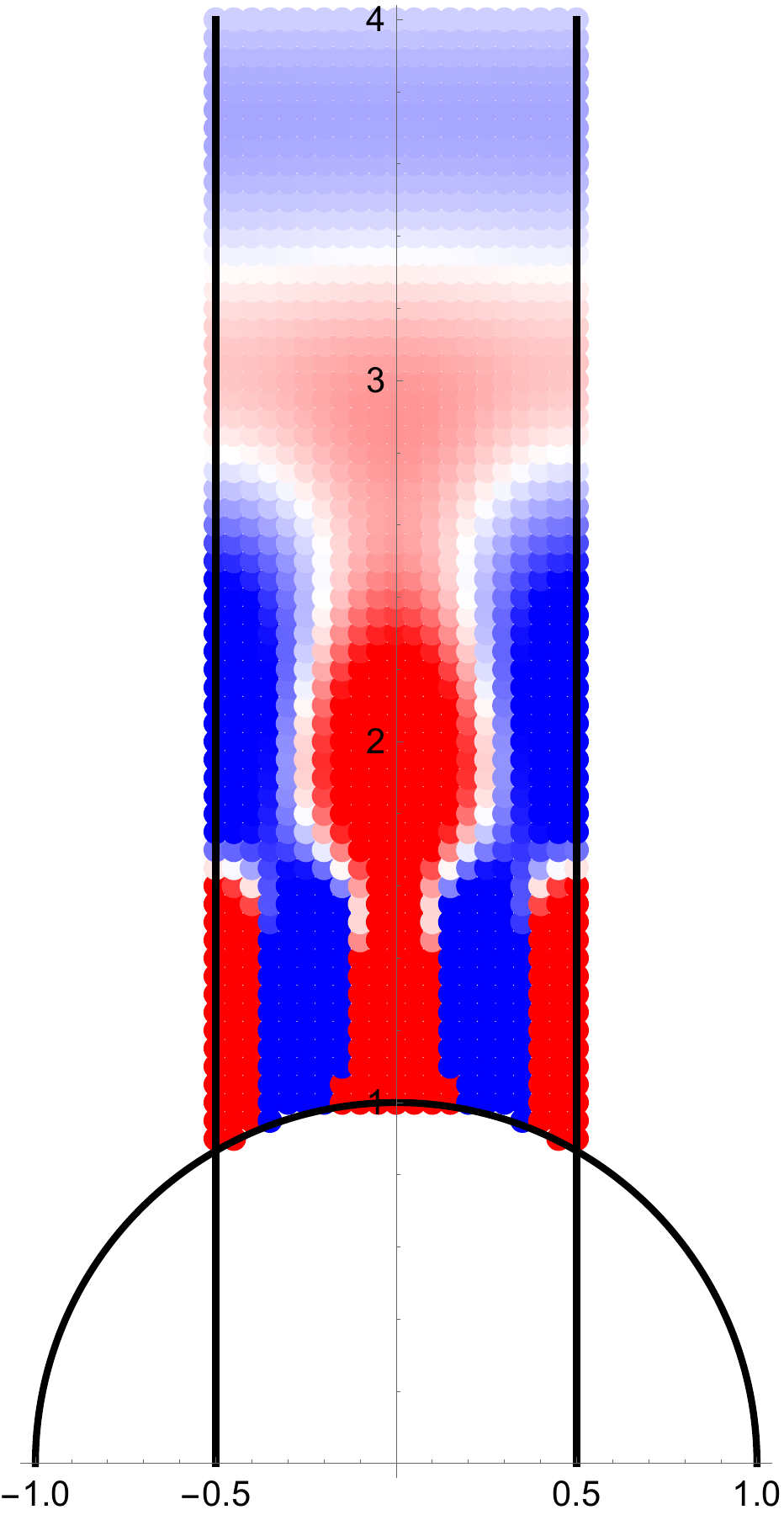}}
  \end{center}  
  \caption{An example of the Eisenstein series with $t=14$. The red dots denote positive regions while the blue dots denote negative regions. Note the  sign change visible along the imaginary axis.}
\label{fig:Eis}
\end{figure}

A heuristic argument of Bogomolny and Schmit \cite{BogomolnySchmit2002} gives a very precise prediction for the asymptotic number $N^{\Omega}(\phi)$ of nodal domains of a Maass cusp form $\phi$ in a compact domain $\Omega\subseteq \Gamma \bk \bH$, namely that $N^\Omega(\phi)$  grows like a constant times $\lambda_\phi$, as $\lambda_\phi \to \infty$. While their prediction is supported by numerics, it seems currently out of reach, and even the weaker claim that $N^{\Omega}(\phi)\to\infty$ as $\lambda_\phi\to \infty$ is 
not currently known unconditionally (and may not be true for general surfaces, see \cite[p.3]{GhoshReznikovSarnak2013}).


The space $\Gamma\bk \bH$ has an orientation reversing isometry, $\sigma(x+iy)=-x+iy$.
We say that a nodal domain is inert if it is preserved by $\sigma$, and split if it is paired with another domain. We denote by $N_{\rm in}(f)$ and $N_{\rm sp}(f)$ the number of inert and split domains. 
Let $\delta\subset \cF_\Gamma$ denote the set of fixed points of $\sigma$, which is naturally partitioned as 
$$\delta=\delta_1\cup\delta_2\cup \delta_3,$$ 
with $\delta_1=\{iy: y\geq 1\}, \delta_2=\{\frac12+iy:y\geq {\sqrt3\over2}\}$ and $\delta_3=\{x+iy: 0<x<\frac12, x^2+y^2=1\}$.
It was then observed in \cite{GhoshReznikovSarnak2013} that for an even cusp form (i.e., a cusp form satisfying $\phi(\sigma z)=\phi(z)$), one can bound $N_{\rm  in}(\phi)$ by counting the number of sign changes of $\phi$ along $\delta$, or more generally, along a non-empty compact segment $\beta\subseteq \delta$. Explicitly, given a segment $\beta\subseteq \delta$, let $K^\beta(\phi)$ denote the number of sign changes of $\phi$ along $\beta$, and $N_{\rm in}^\beta(\phi)$ the number of nodal domains intersecting $\beta$; then 
\begin{equation}\label{eq:Kbeta}
1+\frac12 K^\beta(\phi)\leq N^\beta_{\rm in}(\phi)\leq |Z_\phi\cap \beta|.   
\end{equation}
It is thus possible to reduce the problem of studying the number of (inert) nodal domains to studying the number of sign changes/zeros.
For this problem, 
\cite{GhoshReznikovSarnak2013} proved,
assuming the Lindel\"of hypothesis for the 
$L$-functions attached to $\phi$, 
that,
given a compact geodesic segment $\beta$ in $\delta_1$ or $\delta_2$,   
$$t_\phi^{\nu}\ll |Z_\phi\cap \beta| \ll t_\phi,$$
for any $\nu < 1/12$. (Note that
 the upper bound here is unconditional and follows from general complexification techniques \cite{TothZelditch2009}.) In addition, these techniques can be applied to give a similar, although still conditional, lower bound for the same problem on Eisenstein series. Following this Jang and Jung \cite{JangJung2018} used arithmetic quantum unique ergodicity, to prove qualitatively that the number of nodal domains goes to $\infty$ with the eigenvalue. Moreover, Jung and Young \cite{JungYoung2019} proved an unconditional but weaker lower bound for Eisenstein series with  $\nu < 1/51.$

Recently, Ki \cite[Theorem 1]{Ki2023} proved an essentially sharp (in the exponent) lower bound for both Maass forms and the Eisenstein series, conditional on both the Lindel\"of hypothesis for the associated $L$-function and a fourth moment bound along $\beta$. Explicitly, Ki shows that for any $\vep>0$,
 \begin{equation}\label{eq:KiThm1}
 |Z_f\cap \beta|\ \gg_\vep \ t_f^{1-\vep},   
 \end{equation}
 where $f$ is either a Maass form or the Eisenstein series
 (Ki's technique can also be applied to sign changes, $K^\beta(f)$). Our Theorem \ref{thm:cond Eis} recovers this sharp lower bound for Eisenstein series \emph{without} the assumption of the Lindel\"of hypothesis. 
Moreover, Theorem \ref{thm:cond} 
show similar results for cusp forms, conditioned on an $L^2$ estimate for $L$-functions (namely, Conjecture \ref{con:L2}). 

While we specialize to the modular surface, we can extend this work to congruence subgroups with reflection symmetries. In addition, we specialize our analysis to the central line $z=iy$, but
this
can also be
extended 
to any cuspidal geodesic, see Remark \ref{rmk:geod}. 

\subsection{Main results}

The main goal of
this paper is to prove the same bound as Ki's \eqref{eq:KiThm1} for the Eisenstein series, \emph{without} assuming the Lindel\"of hypothesis.

\begin{theorem}\label{thm:cond Eis}
    Let $\beta=i[a,b]$ be a compact segment of the imaginary line. Let $t_j\to\infty$ and
    suppose that there is some $p>2$   such that for all $\vep>0$,
    \begin{align} \label{Lp bound Eis}
        \left(\int_{a}^b |E_{t_j}(iy)|^p {\rd y\over y}\right)^{1/p} \ll_\vep t_j^\vep.
    \end{align}
    Then 
for any $\vep>0$,
    \begin{align}\label{K bound cond Eis}
        K^\beta(E_{t_j}) \gg_\vep
        t_j^{1-\vep}.
    \end{align}
\end{theorem}

\begin{remark}\label{rmk:epsilon0}
    Explicitly what we show is that the bound of order $t_j^\epsilon$ for the $L^p$ norm, implies a lower bound of order $t_j^{1-\epsilon'}$ for $K^\beta(E_{t_j})$ with any  $\epsilon'>\frac{8p}{p-2}\epsilon$ (see \S \ref{ss:proof 13}). In particular, a sufficiently strong subconvex bound for the sup norm of $E_{t}$ of order $t^\nu$ with $\nu<\frac{1}{8}$ is already sufficient to obtain a non trivial lower bound for  $K^\beta(E_{t})$. We note however that with the current best known bound for the sup norm of $E_t$ we can only take $\nu>\frac{1}{3}$, which is not sufficient to get an unconditional improvement here.  
\end{remark}



The key insight in the proof of Theorem \ref{thm:cond Eis} is to show that, rather than the Lindel\"of hypothesis, one can make do with an estimate on the $L^2$ norm of the $L$-function associated to the Eisenstein series, which translates to a fourth moment estimate on the Riemann zeta function. For Maass forms, we can make the same simplification. However, while the $L^2$ estimate for the associated $L$-function is certainly weaker than the Lindel\"of hypothesis and is known in many instances, it is still not known in the precise setup needed 
in our context.
In fact, such estimates also appear in the study of restricted quantum unique ergodicity for Maass forms and would be of interest there (see \cite{Young2018}). We state the requisite $L^2$ estimate below as Conjecture \ref{con:L2}. Assuming this conjecture holds, we can prove the analogue of Theorem \ref{thm:cond Eis} 
in the context of cusp forms:

\begin{theorem}\label{thm:cond}
    Fix $\beta=i[a,b]$  a compact segment of the imaginary line. Let $\phi_j$ be a sequence of even Hecke cusp forms and 
    assume that there is some $p>2$ such that for 
    any $\vep>0$,
    \begin{align} \label{Lp bound}
        \left(\int_{a}^b |\phi_j(iy)|^p {\rd y\over y}\right)^{1/p} \ll_\vep t_{\phi_j}^\vep.
    \end{align}
    Further, assume that $\phi_j$ satisfy Conjecture \ref{con:L2}.
     Then for any $\vep>0$,
    \begin{align}\label{K bound cond}
        K^\beta(\phi_j) \gg_\vep t_{\phi_j}^{1-\vep}.
    \end{align}
    
\end{theorem}

\begin{remark}\label{rmk:geod}
    As stated, the above theorems concern the geodesic $z=iy$. In fact, the proof below works for any cuspidal geodesic $x+iy$ with $x= \frac{p}{q}$ a rational number. For this, we require estimates on the second moment of the series
    \begin{align*}
        \sum_{n} \frac{a_f(n)e(nx)}{n^s}
    \end{align*}
    and a lower bound on the  $L^2$-norm of the Eisenstein series/cusp form along $\beta=\{x+iy: a<y<b\}$. 

    The lower bound is proved in \cite{Young2018} for Eisenstein series, and in \cite{GhoshReznikovSarnak2013} (although this is only proved for the lines $x=0$, and  $x=\tfrac{1}{2}$) for cusp forms.

    For the estimates on the twisted $L$ series, we split into congruence classes modulo $q$ using Dirichlet characters. This allows us to write the $L$ function as 
    \begin{align*}
        \frac{1}{\phi(q)}\sum_{a \mod q} e_q(aq) \sum_{\chi} \overline{\chi}(a) \sum_{n} \frac{a_f(n)\chi(n)}{n^s}.
    \end{align*}
    Now for cusp forms, bounding the inner twisted $L$-function requires us to extend Conjecture \ref{con:L2} to these. For the Eisenstein series, this requires known estimates for the $4$th moment of Dirichlet $L$-functions \cite{Topacogullari21}. 

\end{remark}

\subsection{Proof strategy}

For both Eisenstein series and Maass forms, the proofs of Theorems \ref{thm:cond}, 
and \ref{thm:cond Eis} 
follow the same strategy. The starting point is \cite[Proof of Theorem 1]{Ki2023}, wherein Ki conditionally proves the inequality \eqref{K bound cond} for all cusp forms (the method also applies to Eisenstein series).


The key idea in our proof is a modification of Ki's argument, allowing us to replace the full strength of the Lindel\"of hypothesis with corresponding bounds on the second moment of the associated $L$-function. For cusp forms, this is Conjecture \ref{con:L2},  while for the Eisenstein series, this boils down to fourth moment estimates on the Riemann zeta function which are well-known (see \S \ref{ss:Eis L}).



\subsection*{Notation}
We use standard Vinogradov notation that $f\ll g$ if there is a constant $C>0$ so that $f(x)\le C g(x)$ for all $x$.

\subsection*{Acknowledgements}
We thank Valentin Blomer, Henryk Iwaniec, and Matt Young for many insightful discussions. Moreover we would like to thank Bingrong Huang, Xiaocheng Li, and Simon Marshall for spotting an error in a preprint. 
This paper was written while the second-named author was visiting Princeton University; he would like to express his gratitude for their hospitality.

\section{Preliminaries}

\subsection{Littlewood's sign changes lemma}

A key analytic ingredient in Ki's proof is \cite[Theorem 2.2]{Ki2023}, which is a variant on a theorem of Littlewood \cite{Littlewood1966} controlling the number of zeros of a real valued function. While Ki's formulation (as well as Littlewood's) discusses the number of zeros, we note that the argument actually controls the number of sign changes. For the sake of completeness, we include the proof of this result below.

Given a real valued function $f$ on the interval $I=[a,b]$ let $M_p(f)$ denote the $L^p(I)$ norm:
\begin{align*}
  M_p(f) : = \left(\frac{1}{|I|} \int_{I} \abs{f(y)}^p \rd y \right)^{1/p}.
\end{align*}

The following is a slight variant of \cite[Theorem 2.2]{Ki2023}.
\begin{lemma} \label{thm:Littlewood}
Let $f$ be a real valued function defined on an open interval containing $I=[a,b]$.
Let $N\in \N$ be sufficiently large so that $f$ is defined on $[a,b+\eta]$ with $\eta=\frac{|I|}{N}$, and define
\begin{align*}
  J(f, \eta) = \frac{1}{|I|} \int_I \abs{ \int_0^\eta f(y+ v) \rd v } \rd y.
\end{align*}
Suppose that there is some $c\in(0,1)$ such that $M_1(f)\ge c M_2(f)$ and that
  $J(f, \eta) < \frac{c^3 \eta M_2(f)}{16}$.
  Then the number of sign changes, $K^I(f)$, of $f$ on $I$ satisfies 
  $$
  K^{I}(f) \ge \frac{c^2 N}{8}.
  $$
\end{lemma}
\begin{proof}
By scaling and shifting $f$ we may assume that $I=[0,1]$ and $\eta=\frac{1}{N}$.
For any $1\leq m\leq N$ let $I_m=[\frac{m-1}{N},\frac{m}{N})$, and define 
$$J_m(f,\eta):=\int_{I_m}\left|\int_0^\eta f(y+t)dt\right|dy.$$ 
Let $\cM_1=\{1\leq m\leq N: \mbox{ $f$ changes sign in $I_m$}\}$
and let $\cM_2$ be its complement. Since in any interval $I_m$ with $m\in \cM_1$ there is at least one sign change of $f$, we have that $K^I(f)\geq |\cM_1|=N-|\cM_2|$.
Let $E=\cup_{m\in \cM_2} I_m$ so that $|E|=\frac{|\cM_2|}{N}$ and the result will follow by showing that $|E|<1-\frac{c^2}{8}$. We assume now that $|E|>1-\frac{c^2}{8}$ and proceed by contradiction. 

Let $H:=\{y\in I: |f(y)|\geq \frac{c M_2(f)}{2}\}$, then the assumption $M_1(f)\ge c M_2(f)$ implies that $|H|\geq 1-\frac{c^2}{4}$.
Indeed, we can estimate 
$$c M_2(f)\ \leq\  M_1(f)\ \leq \ \int_{H}f+\int_{H^c}f \ \leq \ |H|^{1/2}M_2(f)+(1-|H|)\frac{c M_2(f)}{2}.$$
Setting $X=\sqrt{|H|}$, then from the above display we see that $X^2-\frac{2}{c}X+1\leq 0$ hence $X> \frac{c}{2}$ so $|H|>\frac{c^2}{4}$.
For any $m\in \cM_2$ we have that   
$J_m^*(f,\eta):=\int_{I_m}\int_0^\eta |f(y+t)|dtdy=J_m(f,\eta)$. We can estimate on one hand 
$$\sum_{m\in \cM_2}J_m^*(f,\eta)=\sum_{m\in \cM_2}J_m(f,\eta)\leq J(f,\eta)<\frac{c^3 \eta M_2(f)}{16}.$$
On the other hand, we have 
$$\sum_{m\in \cM_2}J_m^*(f,\eta)=\int_0^\eta \int_E |f(y+t)|dydt=\int_0^\eta \int_{E_t} |f(y)|dydt,$$
where $E_t$ is the shift of $E$ by $t$. By our assumption $|E_t|=|E|>1-\frac{c^2}{8}$ and since 
$$\frac{c^2}{4}\leq |H|=|H\cap E_t|+|H\cap E_t^c|<|H\cap E_t|+\frac{c^2}{8},$$ using our bounds on $|H|$, we also have that $|H\cap E_t|>\frac{c^2}{8}$.
We can thus bound
$$\int_0^\eta \int_{E_t} |f(y)|dydt\geq \int_0^\eta \int_{E_t\cap H} |f(y)|dydt>  \frac{c^3 \eta M_2(f)}{16},$$
in contradiction.
\end{proof}

\subsection{Preparation for Eisenstein series}
\label{ss:Eis L}
We now collect a number of results regarding the Eisenstein series and its $L$-function that will be needed in our proof.

    


The first result we need is a lower bound on $M_2(E_t)$ given in \cite{Young2018} that is needed in order to apply Lemma \ref{thm:Littlewood}.
\begin{proposition}[{\cite[Theorem 1.1]{Young2018}}]\label{prop:lower bound}
\begin{align*}
    M_2(E_t) : = \left(\frac{1}{b-a} \int_a^b \abs{E_t(iy)}^2 \rd y \right)^{1/2} \gg (\log T)^{1/2}
\end{align*}
for any $t\in \R$ and any fixed segment $(a,b)$. 
\end{proposition}

The second result we need is about the size of the $L$-function of the Eisenstein series on the critical line, which can be written explicitly in terms of the Riemmann zeta function. Recall, the Lindel\"of hypothesis predicts that, for any $\vep >0$ and all $t \in \R$, one has $|\zeta(\frac{1}{2} + it)|= O((1+|t|)^{\vep})$. While the Lindel\"of hypothesis is far from reach of modern technology, there are some results concerning moment bounds on the zeta function which will suffice for our purposes. The following classical theorem was proven by Heath-Brown
\begin{theorem}[\cite{HeathBrown1979}]
    There is $\kappa>0$ such that for any $T$ large one has
    \begin{align}\label{4 moment}
        \frac{1}{T} \int_0^{T} \abs{\zeta(\frac{1}{2}+ it)}^{4} \rd t =  P_4(\log(T))+O(T^{-\kappa}),
    \end{align}
    with $P_4(x)$ a polynomial of degree $4$.
\end{theorem}

\subsection{Preparation for Maass forms}\label{ss:maaass}
We now collect the corresponding results we need to apply the argument for Maass forms. 

Once again, the lower bound we need for $M_2(\phi)$ is known, this time having been proved by Ghosh, Reznikov and Sarnak \cite{GhoshReznikovSarnak2013}. 

\begin{proposition}[{\cite[Theorem 6.1]{GhoshReznikovSarnak2013}}]\label{prop:lower bound cusp}
\begin{align*}
    M_2(\phi) : = \left(\frac{1}{b-a} \int_a^b \abs{\phi(iy)}^2 \rd y \right)^{1/2} \gg 1
\end{align*}
for any segment $(a,b)$. 
\end{proposition}

The final ingredient we need is an estimate for the $L$-function associated to the cusp for $\phi$, we now describe.
Given a cusp for $\phi$, we consider the Fourier expansion
\begin{align*}
    \sum_{n \neq 0} \rho_\phi(n) y^{1/2} K_{it_\phi} (2\pi \abs{n} y) e(nx),
\end{align*}
where $K$ is the $K$-Bessel function. Furthermore, we let $\lambda_\phi(n) = \frac{\rho_\phi(n)}{\rho_\phi(1)}$ denote the eigenvalues of the Hecke operators.

With the Fourier coefficients in hand we define the associated $L$-function
\begin{align} \label{L cusp}
       L_{\phi}(s) : = \sum_{n=1}^\infty \frac{\lambda_\phi(n)}{n^s}.
\end{align}
The following conjecture gives a mean square bound for this $L$-function. 
\begin{Con} \label{con:L2}
    Let $\phi$ be a Maass form with spectral parameter $t_\phi$. There exists a $\delta>0$ such that, for $2T\le t_\phi \le T^{1+\delta}$ and every $\vep>0$, we have
\begin{align}\label{L cusp bound}
    \frac{1}{T}\int_T^{2T} \abs{L_{\phi}(\frac{1}{2}+it)}^2 \rd t \ll t_\phi^\vep,
\end{align}
as $T\to\infty$.
\end{Con}

Such an estimate clearly follows from the Lindel\"of hypothesis, and we note that for the range $2T > t_\phi$ the estimate \eqref{L cusp bound} is known (see \cite[Section 6.1]{GhoshReznikovSarnak2013}). While it is possible that our range $2T\le t_\phi \le T^{1+\delta}$ is also within reach of current technology we were not able to establish it and thus leave it as an open conjecture.

\section{Proof for Eisenstein series}

We start by proving Theorem \ref{thm:cond Eis}. The proof for cusp forms is more or less identical; we explain the major differences in \S \ref{s:cusp proof}. The proof for both is an application of Theorem \ref{thm:Littlewood}  for which we require a lower bound on $M_2(\cdot)$ (see Proposition \ref{prop:lower bound}) and an upper bound on $J(\cdot)$.

\subsection{Upper bound on $J$}
Rather than work with $E_t(z)$ it is more convenient to work with
\begin{align*}
    f_{t}(z) = \frac{1}{\sqrt{y}}E_t(z),
\end{align*}
since $y$ is bounded away from $0$ and $\infty$, any statement about zeroes or nodal lines for $f_t$ holds equally well for $E_t$.
Thanks to Theorem \ref{thm:Littlewood}, our goal is now to bound
\begin{align}
    J(f_t,\eta) := \frac{1}{b-a}\int_a^b \abs{\int_0^\eta f_{t}(i(y+v)) \rd v}\rd y.
\end{align}

\begin{proposition} \label{prop:J bound}
    Fix an interval $(a,b) \subset \R_{>0}$ for all $\frac{2}{t}<\eta <1$ and $t\geq 10$ sufficiently large
    \begin{align}
        J(f_t,\eta)  \ll  \eta \left ( \frac{ \log(t)^{9}}{\sqrt{\eta t}} +\frac{ \log(t)^{7}}{t^{\kappa/2} }\right)+ \frac{(\log(t))^{9}}{t}
    \end{align}
\end{proposition}

\begin{proof}

First, Fourier expand the Eisenstein series: \cite[(3.20)]{Iwaniec2002}
\begin{align*}
    E(z,s) = y^s + \varphi(s) y^{1-s} + \frac{4 \sqrt{y}}{\theta(s)} \sum_{n=1}^\infty \eta_{s-1/2}(n) K_{s-1/2}(2\pi n y) \cos(2\pi n x)
\end{align*}
with $\theta(s) = \pi^{-s} \Gamma(s) \zeta(2s)$ and $\varphi(s) = \theta(1-s) \theta(s)^{-1}$, and where
\begin{align*}
    \eta_t(n) = \sum_{ab=n} \left(\frac{a}{b}\right)^t.
\end{align*}

With that, we define the $L$-function
\begin{align*}
    L(t,\nu)  &= \sum_{n \ge 1} \frac{\eta_{it}(n)}{n^\nu}
\end{align*}
It's well-known that this $L$-function can be related to the Riemann zeta function:
\begin{align}
    L(t,\nu) &= \sum_{n \ge 1} \frac{1}{n^{\nu+it}}\sum_{d| n} d^{2it}  \notag \\
    &= \sum_{d \ge 1}  d^{2it}   \sum_{n \equiv 0 \mod d } \frac{1}{n^{\nu+it}}  \notag \\
    &= \sum_{d \ge 1}  d^{2it}   \sum_{n \ge 1} \frac{1}{n^{\nu+it} d^{\nu+it}} 
    =\zeta(\nu+it)\zeta(\nu-it). \label{L to zeta}
\end{align}

Following  \cite[Proof of Lemma 4.1]{Ki2023} we can relate $J(f_t,\eta)^2$ to this $L$-function. Specifically, we can write
\begin{align*}
    J(f_t,\eta)^2 &= \left(\frac{1}{b-a}\int_a^b \abs{\int_0^\eta f_t(i(y+v))\rd v}\rd y\right)^2 \\   
    &\ll \left(\frac{1}{b-a}\int_a^b \abs{ 
  \frac{(y+\eta)^{1+it} - y^{1+it}}{1+it} +\varphi(\frac12+it)\frac{(y+\eta)^{1-it} - y^{1-it}}{1-it} } \rd y\right)^2 \\ 
  &\phantom{++} + \left(\frac{1}{b-a}\int_a^b \abs{ \int_0^\eta \left[\frac{4 }{\theta(s)} \sum_{n=1}^\infty \eta_{it}(n) K_{it}(2\pi n (y+v)) e(nx) \right] \rd v} \rd y\right)^2\\
   &\ll \frac{1}{t^{2}} + \cJ(t),
\end{align*}
where
\begin{align*}
    \cJ(t) = \left(\frac{1}{b-a}\int_a^b \abs{ \int_0^\eta \left[\frac{4 }{\theta(s)} \sum_{n=1}^\infty \eta_{it}(n) K_{it}(2\pi n (y+v)) e(nx) \right] \rd v} \rd y\right)^2.
\end{align*}

From here we can expand the $K$-Bessel function \cite[(10.32.13)]{Olver}, that is,
\begin{align*}
    K_{it}(z) = \frac{(z/2)^{it}}{4\pi i} \int_{(c)} \Gamma(\nu) \Gamma(\nu-it) \left(\frac{z}{2}\right)^{-2\nu} \rd \nu,
\end{align*}

\noindent and set $c=1/4$, yielding
\begin{align*}
    \cJ(t) &\ll \int_a^b \abs{ \int_0^\eta \int_{(1/4)}  \left[\frac{ (y+v)^{-2\nu+it}}{\theta(\frac{1}{2} + it)} \Gamma(\nu) \Gamma(\nu-it) \sum_{n=1}^\infty \eta_{it}(n)   \left(\pi n\right)^{-2\nu+it}  \right] \rd \nu\rd v}^2 \rd y\\
    &= \int_a^b \abs{ \int_0^\eta \int_{(1/4)}  \left[\frac{ (y+v)^{-2\nu}}{\theta(\frac{1}{2} + it)} \Gamma(\nu+\frac{it}{2}) \Gamma(\nu-\frac{it}{2}) \sum_{n=1}^\infty \eta_{it}(n)   \left(\pi n\right)^{-2\nu}    \right] \rd \nu\rd v}^2 \rd y \\
    &\ll \int_a^b  \int_{(1/2)} \abs{I(\eta,\nu,y) \gamma(\nu,t) L(t,\nu)}^2  \rd \nu \rd y
\end{align*}
where $\gamma(\nu,t) = \frac{ \Gamma(\frac{\nu+it}{2}) \Gamma(\frac{\nu- it}{2})}{\theta(\frac{1}{2} + it)} \pi^{-\nu} $ and $I(\eta,y;\nu) : = \int_0^\eta (y+v)^{-\nu} \rd v$. 
We now estimate the inner integral. 
Write $\nu = 1/2+ ir$, and using the invariance under  $r\mapsto -r$ it is enough to estimate the integral 
$$ \int_0^\infty \abs{I(\eta,y,\frac12+ir) \gamma(\frac12+ir,t) L(t,\frac12+ir)}^2  \rd r.$$
Noting that $\eta<1$ and that the interval $(a,b)$ is fixed, we can bound the integral, $I$ by
\begin{align}\label{I bound}
    I(\eta,y; \tfrac12+ir) = \int_0^\eta (y+v)^{-1/2} e^{ -i \log(y+v)r} \rd v
    \ll
    \min(\eta, \frac{1}{|r|}).
\end{align} 
Using  Stirling's formula, the $\gamma$-factor can be bounded by
\begin{align*}
    \gamma(\frac12+ir,t) &= \frac{ \Gamma(\frac{1/2 + ir +it}{2}) \Gamma(\frac{1/2 + ir -it}{2})}{\theta(\frac{1}{2} + it)} \pi^{-1-ir}\notag \\
    &\ll \frac{ e^{-\pi |t+r|/4} e^{-\pi \abs{r-t}/4}}{e^{-\pi t/2} \zeta(1+2it)} \frac{1}{((1+\abs{r-t})(r+t))^{1/4}}\\
    &\ll  (\log(t))^7 \frac{ e^{-\pi |t+r|/4} e^{-\pi \abs{r-t}/4}e^{\pi t/2} }{ ((1+\abs{r-t})(1+|r+t|))^{1/4}}
    , \label{gamma bound}
\end{align*}
where we used the bound $\zeta(1+2it) \gg \frac{1}{\log(t)^7}$  (see \cite[(3.6.5)]{Titchmarsh1951}).

First when $r\geq t$ we can bound 
$$ | I(\eta,y; \tfrac12+ir) \gamma(\frac12+ir,t) |^2 \ll \log(t)^{14} \frac{e^{-\pi (r-t)}}{((1+(r-t))(1+(r+t))^{1/2} r^2},$$
and using the convexity bound $\zeta(\frac{1}{2} + it) \ll t^{1/4}$ for the zeta function we can bound 
$$|L(t,\frac12+ir)|^2=|\zeta(\frac12+i(t+r))\zeta(\frac12+i(r-t))|^2\ll ((1+(r-t))(1+(r+t))^{1/2},$$
hence in this range 
 $$| I(\eta,y; \tfrac12+ir) \gamma(\frac12+ir,t) |^2\ |L(t,\frac12+ir)|^2\ \ll \  t^{-2} \log(t)^{14} e^{-\pi (r-t)},$$
 and we can bound 
\begin{align}
\int_t^\infty \abs{I(\eta,y,\frac12+ir) \gamma(\frac12+ir,t) L(t,\frac12+ir)}^2  \rd r &\ll \frac{\log(t)^{14}}{t^2}.
\end{align}
Next for the range $r\leq \frac{1}{\eta}\leq \frac{t}{2}$ we can bound 
$$ | I(\eta,y; \tfrac12+ir) \gamma(\frac12+ir,t) |^2 \ll  \frac{\eta^2 \log(t)^{14}}{t},$$
to get
\begin{align*}
 \int_0^{1/\eta} \abs{I(\eta,y,\frac12+ir) \gamma(\frac12+ir,t) L(t,\frac12+ir)}^2  \rd r \ll  \frac{\eta^2 \log(t)^{14}}{t}\int_0^{1/\eta} |L(t,\frac12+ir)|^2\rd r.\\
\end{align*}
Now use Cauchy-Schwarz for the inner integral together with \eqref{4 moment} to bound 
\begin{align*}
\int_0^{1/\eta} |L(t,\frac12+ir)|^2\rd r&  \ll (\int_0^{1/\eta}| \zeta(\frac12+i(t-r))|^4\rd r)\int_0^{1/\eta}|\zeta(\frac12+i(t+r)|^4\rd r)^{1/2}\\
 &\ll \int_{t-1/\eta}^{t+1/\eta} |\zeta(\frac12+ir)|^4dt\\
 &\ll (t+1/\eta)P_4(\log(t+1/\eta)))-(t-1/\eta)P_4(\log(t-1/\eta))+O(t^{1-\kappa})\\
 &\ll  \frac{\log(t)^4}{\eta}+t^{1-\kappa}
\end{align*}
to conclude that 
\begin{align*}
 \int_0^{1/\eta} \abs{I(\eta,y,\frac12+ir) \gamma(\frac12+ir,t) L(t,\frac12+ir)}^2  \rd r\ll  \eta^2\left(\frac{ \log(t)^{18}}{\eta t} +\frac{ \log(t)^{14}}{t^\kappa}\right).
\end{align*}
Finally, in the range $\frac{1}{\eta}\leq r\leq t$ we first bound
$$ | I(\eta,y; \tfrac12+ir) \gamma(\frac12+ir,t) |^2 \ll  \frac{ (\log(t))^{14} }{r^2((1+(t-r))(1+t+r))^{1/2}},$$
hence
\begin{align*}
 \int_{1/\eta}^t \abs{I(\eta,y,\frac12+ir) \gamma(\frac12+ir,t) L(t,\frac12+ir)}^2  \rd r&\ll  (\log(t))^{14} \int_{1/\eta}^t \frac{ |L(t,\frac12+ir)|^2}{r^2 ((1+(t-r))(1+t+r)))^{1/2}}\rd r\\
 &\ll \frac{ (\log(t))^{14}}{\sqrt{t}} \int_{0}^{t-1/\eta} \frac{|L(t,\frac12+i(t-r))|^2}{ (t-r)^2(1+r)^{1/2}}\rd r.\\
\end{align*}
Split the integral into dyadic intervals to estimate 
\begin{align*}
  \int_{0}^{t-1/\eta} \frac{|L(t,\frac12+i(t-r))|^2}{ (t-r)^2(1+r)^{1/2}}\rd r &\ll  t^{-2}\int_0^1 |L(t,\frac12+i(t-r))|^2 \rd r\\
  &+\sum_{k=1}^{\log(t-1/\eta)}\frac{1}{2^{k/2}(t-2^k)^2}\int_{2^{k-1}}^{2^k}|L(t,\frac12+i(t-r))|^2\rd r\\
\end{align*}
We can bound the first integral by 
\begin{align*}
\int_0^1 |L(t,\frac12+i(t-r))|^2 \rd r&=\int_0^1 |\zeta(\frac{1}{2}+ir)|^2|\zeta(\frac{1}{2}+i(2t-r)|^2 \rd r\\
& \leq\left( \int_0^1 |\zeta(\frac{1}{2}+ir)|^4 \rd r\int_{2t-1}^{2t} |\zeta(\frac{1}{2}+ir|^4 \rd r \right)^{1/2}\\
& \ll  t^{1/2} (\log t)^2
\end{align*}
and for each dyadic interval with $A=2^k\leq t$ we have
\begin{align*}
\int_{A}^{2A}|L(t,\frac12+i(t-r))|^2\rd r &\ll \int_{A}^{2A} |\zeta(\frac12+i(2t-r)|^2|\zeta(\frac{1}{2}+ir)|^2\rd r\\
&\ll \left(\int_{2t-2A}^{2t-A} |\zeta(\frac12+ir|^4 \rd r \int_{A}^{2A}|\zeta(\frac{1}{2}+ir)|^4\rd r\right)^{1/2}\\
&\ll (A\log^4(t)+t^{1-\kappa})^{1/2}(A\log^4(t))^{1/2}\ll \log^4(t)(A+t^{\frac{1-\kappa}{2}}A^{1/2}).
\end{align*}

Hence 
\begin{align*}
\int_{0}^{t-1/\eta} \frac{|L(t,\frac12+i(t-r))|^2}{ (t-r)^2(1+r)^{1/2}}\rd r&  \ll  t^{-3/2}(\log t)^2+(\log(t))^4 \sum_{k=1}^{\log(t-1/\eta)}\frac{2^{k/2} +t^{\frac{1-\kappa}{2}}}{(t-2^k)^2}\\
 &\ll  t^{-3/2}+ (\log t)^4\int_1^{\log(t-1/\eta)}\frac{2^{u/2} +t^{\frac{1-\kappa}{2}}}{(t-2^u)^2}du\\
 &\ll  \frac{ (\log t)^4}{ t^{3/2}},
 \end{align*}
 and 
\begin{align*}
 \int_{1/\eta}^t \abs{I(\eta,y,\frac12+ir) \gamma(\frac12+ir,t) L(t,\frac12+ir)}^2  \rd r &\ll \frac{ (\log(t))^{18}}{t^2}.\end{align*}
Combining the three terms and integrating over the outer interval $(a,b)$  we get that 
$$J(f_t,\eta)^2\ll \eta^2\left(\frac{ \log(t)^{18}}{\eta t} +\frac{ \log(t)^{14}}{t^\kappa}\right)+ \frac{(\log(t))^{18}}{t^2},$$
and taking a square root concludes the proof.
\end{proof}



\subsection{Proof of Theorem \ref{thm:cond Eis}} \label{ss:proof 13}%
Let $t_j\to \infty$ be a sequence for which we have we have the $L^p$ bounds
$M_p(f_{t_j})\ll_\epsilon t_j^\epsilon$ and use $L^p$ interpolation to bound 
$$M_2(f_{t_j})^2\leq M_1(f_{t_j})^{\frac{p-2}{p-1}}M_p(f_{t_j})^{\frac{p}{p-1}}.$$
This combined with the uniform lower bound $M_2(f_t)\gg 1$ implies that there is a constant $C_1=C_1(\epsilon)$ so that
$$M_1(f_{t_j})\geq C_1 M_2(f_{t_j}) t_j^{-\frac{\epsilon p}{p-2}}.$$
Let $c=C_1t_j^{-\frac{ p\epsilon }{p-2}}$ and $\eta=\frac{1}{N}=t_j^{\delta-1}$,
so that $c^3\eta M_2(f_{t_j})\geq C_2 \eta t_j^{-\frac{3 p\epsilon}{p-2}}$. 
From the upper bound
$$J(f_{t_j},\eta)\leq C_5\eta \log(t_j)^9t_j^{-\delta/2},$$
we see that $J(f_{t_j},\eta)\leq \frac{c^3}{16}\eta M_2(f_{t_j})$ as long as $\delta>\frac{6 p\epsilon}{p-2}$ in which case Theorem \ref{thm:Littlewood} implies that
$$K^\beta(f_{t_j})\geq C_3 t_j^{1-\delta-\frac{2 p \epsilon }{p-2}}$$
for an appropriate constant $C_3>0$. 
In particular, we see that for any $\kappa>8$, for all sufficiently large $t_j$ we have that 
$$K^\beta(f_{t_j})\geq  t_j^{1-\frac{\kappa p \epsilon }{p-2}},$$
from which the claim follows.

\section{Proof for cusp forms}
\label{s:cusp proof}

 The proof for cusp forms follows along nearly identical lines. Once again, to apply Theorem \ref{thm:Littlewood} we require a lower bound on $M_2(\cdot)$ (see Proposition \ref{prop:lower bound cusp} and an upper bound on $J(\cdot)$).

\subsection{Upper bound on $J$}

For the bound on $J(\phi,\eta)$ we again renormalize
\begin{align*}
    f(z) = \frac{1}{\sqrt{y}}\phi(z).
\end{align*}
Hence our goal is to bound
\begin{align}
    J(f,\eta)  := \frac{1}{b-a}\int_a^b \abs{\int_0^\eta f(i(y+v)) \rd v}\rd y,
\end{align}
as follows.

\begin{proposition} \label{prop:J bound}
Assuming Conjecture \ref{con:L2}, for any compact interval $(a,b) \subset \R_{>0}$ and $\eta \in (\frac{2}{t},1)$, for any $\vep>0$ we have that
    \begin{align}
        J(f,\eta) \ll  \eta  \frac{t_\phi^\vep}{\sqrt{\eta t_\phi}}+t^{\vep-1} .
    \end{align}
 \end{proposition}
\begin{proof}
As for the Eisenstein series, we can again Fourier expand the Maass form, 
$$f(iy)= \sum_{n \neq 0} \rho_\phi(n)  K_{it_\phi} (2\pi \abs{n} y),$$
with $\rho_\phi(n)=\rho_\phi(1)\lambda_\phi(n)$,
and use the integral equation of the $K$-Bessel function to relate $J(f,\eta)$ to the $L$-function \eqref{L cusp}.
That is, we have that
\begin{align*}
    J(f,\eta)^2 &=  \left(\frac{1}{b-a}\int_a^b \abs{ \int_0^\eta \left[\rho_\phi(1)\sum_{n=1}^\infty \lambda_\phi(n) K_{it}(2\pi n (y+v)) e(nx) \right] \rd v} \rd y\right)^2 \\
    &\ll \int_a^b  \int_{(1/2)} \abs{I(\eta,y; \nu) \gamma(\nu,t) L_{\phi}(\nu)}^2  \rd \nu \rd y
\end{align*}
where $\gamma(\nu,t) = \rho_\phi(1)  \Gamma(\frac{\nu+it}{2}) \Gamma(\frac{\nu- it}{2}) \pi^{-\nu} $ and $I(\eta,y; \nu) : = \int_0^\eta (y+v)^{-\nu} \rd v$. 
We have the bound  \eqref{I bound} for $I(\eta,y; \frac12+ir)$ as before and 
using the bound $\rho_\phi(1) \ll t_\phi^\vep e^{\frac{\pi t_\phi}{2}}$ \cite[(14)]{GhoshReznikovSarnak2013}, and Stirling's formula, we can similarly bound
$$
    \gamma(\frac12+ir,t) \ll  t_\phi^\vep \frac{ e^{-\pi |t_\phi+r|/4} e^{-\pi \abs{r-t_\phi}/4}e^{\pi t/2} }{ ((1+\abs{r-t_\phi})(1+|r+t^\phi|))^{1/4}}.$$
    
    We can again reduce the inner integral to the range $0<r<\infty$ and split it into three ranges 
   $$ \int_0^\infty \abs{I(\eta,y; \frac12+ir) \gamma(\frac12+ir,t) L_{\phi}(\frac12+ir)}^2  \rd r= \cI_0^{1/\eta}+\cI_{1/\eta}^{t_\phi}+\cI_{t_\phi}^\infty.$$
 For the last  the range $r\geq t_\phi$, we can use the convexity bound  $|L_\phi(\frac12+ir)|\ll (1+r+t_\phi)^{1/4+\vep}$ (see, e.g., \cite{IwaniecSarnak2000}) to estimate 
  \begin{align*}
 \cI_{t_\phi}^\infty
 & \ll t_\phi^{\vep-2}  \int_{t_\phi}^\infty  \frac{ e^{-\pi (r-t_\phi)} (1+r+t_\phi)^{1/2+\vep}  }{ ((1+\abs{r-t_\phi})(1+|r+t^\phi|))^{1/2} }  \rd r \\
  & \ll t_\phi^{\vep-2}  \int_{0}^\infty  \frac{ e^{-\pi r} (1+r+2t_\phi)^{1/2+\vep}  }{(1+ r)(r+2t^\phi|))^{1/2} }  \rd r\ll  t_\phi^{\vep-2}.
   \end{align*}

 In the first range when $r\leq 1/\eta$, we have that
  \begin{align*}
 \cI_0^{1/\eta}
 & \ll t_\phi^{\vep-1} \eta^2 \int_{0}^{1/\eta}  |L_{\phi}(\frac12+ir)|^2   \rd r \\
 &\ll  \eta^2  \frac{t_\phi^\vep}{\eta t_\phi}. \end{align*}
 
 Finally, for $1/\eta<t<t_\phi$, split to dyadic intervals and apply Conjecture \ref{con:L2}:
  \begin{align*}
 \cI_{1/\eta}^{t_\phi} 
 & \ll t_\phi^{\vep}  \int_{1/\eta}^{t_\phi} \frac{| L_{\phi}(\frac12+ir)|^2}{((r+t_\phi)(t_\phi-r) )^{1/2} r^2}  \rd r \\
 &\ll  t_\phi^{\vep} \sum_{k=\log(1/\eta)}^{\log(t_\phi)}  \frac{1}{((2^k+t_\phi)(t_\phi-2^k) )^{1/2} 2^{2k}}  \int_{2^{k-1}}^{2^k}| L_{\phi}(\frac12+ir)|^2\rd r\\
 &\ll  t_\phi^{2\vep} \sum_{k=\log(1/\eta)}^{\log(t_\phi)}  \frac{1}{((2^k+t_\phi)(t_\phi-2^k) )^{1/2} 2^{k}}\\
 &\ll  t_\phi^{2\vep}\int_{\log(1/\eta)}^{\log(t_\phi)}  \frac{1}{((2^u+t_\phi)(t_\phi-2^u) )^{1/2} 2^{u}}du\\
 &\ll  t_\phi^{2\vep}\int_{1/\eta}^{t_\phi}  \frac{1}{((v+t_\phi)(t_\phi-v) )^{1/2} v^2}dv\\
 &\ll t_\phi^{2\vep-2}\int_{1/{\eta t}}^{1}  \frac{1}{ ((1+v)(1-v) )^{1/2} v^2 }dv\ll  t_\phi^{2\vep-2}.
 \end{align*}
 
Integrating over $(a,b)$ we see that 
$$J(f,\eta)^2\ll \eta^2 \frac{t_\phi^\vep}{\eta t_\phi}+t_\phi^{\vep -2},$$ 
and taking square roots concludes the proof.
\end{proof}

\begin{proof}
    Let $\phi_j$ be a sequence of cusp and let $f_j(z)=y^{-1/2}\phi_j(z)$ and let $t_j=t_{\phi_j}$. Assume that we have we have the $L^p$ bounds
$M_p(f_{t_j})\ll_\epsilon t_j^\epsilon$ and that Conjecture \ref{con:L2} holds for these cusp forms. As before combining $L^p$ interpolation 
with the lower bound $M_2(f_t)\gg 1$ implies that there is a constant $C_1=C_1(\epsilon)$ so that
$$M_1(f_{t_j})\geq C_1 M_2(f_{t_j}) t_j^{-\frac{\epsilon p}{p-2}}.$$
Let $c=C_1t_j^{-\frac{ p\epsilon }{p-2}}$ and $\eta=\frac{1}{N}=t_j^{\delta-1}$,
so that $c^3\eta M_2(f_{t_j})\geq C_2 \eta t_j^{-\frac{3 p\epsilon}{p-2}}$. 
From the upper bound
$$J(f_{t_j},\eta)\ll  \eta  \frac{t_j^\vep}{\sqrt{\eta t_j}},$$
we see that $J(f_{t_j},\eta)\leq \frac{c^3}{16}\eta M_2(f_{t_j})$ as long as $\delta>\frac{(7 p-2)\epsilon}{p-2}$ in which case Theorem \ref{thm:Littlewood} implies that
$$K^\beta(f_{t_j})\geq C_3 t_j^{1-\delta-\frac{2p \epsilon }{p-2}}$$
for an appropriate constant $C_3>0$. 
In particular, we see that for any $\kappa>9$, for all sufficiently large $t_j$ we have that 
$$K^\beta(f_{t_j})\geq  t_j^{1-\frac{\kappa\epsilon p }{p-2}},$$
from which the claim follows.
\end{proof}


\bibliographystyle{alpha}

\bibliography{DKbibliog}
\end{document}